\documentclass[12pt]{article}

\usepackage{amsfonts}
\usepackage{amssymb}
\usepackage{amsthm}
\usepackage{amsmath}
\usepackage{mathrsfs}
\usepackage{stmaryrd}
\usepackage[latin1]{inputenc}
\usepackage[all]{xy}

\author{Alberto G. Raggi-C\'ardenas\footnote{Partially supported by a grant from CONACyT, Project B0291  \textit{Funtores de tipo Burnside}.}$\, $  and Nadia Romero$^*$\\
\begin{normalsize}
Instituto de Matem\'aticas, UNAM, A.P. 61-3,\end{normalsize}\\
\begin{normalsize}
C.P. 58089, Morelia, Michoac\'an, M\'exico.\end{normalsize}\\
\begin{normalsize}
E-mail: \texttt{graggi@shi.matmor.unam.mx}, \texttt{nadia@matmor.unam.mx}\end{normalsize}
}
\title{On primordial groups for the Green ring}
\date{ }

\newcommand{\pp}{\textrm{Prim}}
\newcommand{\gal}{\mathcal Gal}

\newcommand{\zze}{\mathbb Z}

\newcommand{\triv}{\textrm{triv}}

\newcommand{\ov}[1]{\overline{#1}}

\newcommand{\gr}{\begin{large}\textit{a}\end{large}}
\newcommand{\ind}[2]{\!\uparrow_{#1}^{#2}}
\newcommand{\res}[2]{\downarrow_{#1}^{#2}}

\theoremstyle{plain}
\newtheorem{thm}{Theorem}[section]
\newtheorem{pro}[thm]{Proposition}
\newtheorem{coro}[thm]{Corollary}
\newtheorem{lemma}[thm]{Lemma}

\theoremstyle{definition}
\newtheorem{defi}[thm]{Definition}
\newtheorem{nota}[thm]{Notation}

\theoremstyle{remark}

\begin{document}

\maketitle

\begin{abstract} 
Consider the Mackey functor assigning to each finite group $G$ the Green ring of finitely generated $kG$-modules, where $k$ is a field of characteristic $p>0$.  Th\'evenaz foresaw in 1988 that the class of primordial groups for this functor is the family of $k$-Dress groups. In this paper we prove that this is true for the subfunctor defined by the Green ring of finitely generated $kG$-modules of trivial source.
\end{abstract}


\section{Introduction}
For a field $k$, the Green ring of the category of finitely generated $kG$-modules, $\gr(kG)$, is by definition, spanned over $\zze$ by elements $[M]$, one for each isomorphism class of finitely generated $kG$-modules and with structures given by $[M]+[N]=[M\oplus N]$ and $[M][N]=[M\otimes_k N]$. The subring generated by the $kG$-modules of trivial source (defined in Section 2.1) is denoted by $\gr(kG,\, \triv)$.

Assigning to each finite group $G$ either of the two above-mentioned rings defines a  
globally defined Mackey functor, as described in Bouc \cite{bouc} and Webb \cite{webb}. These functors are denoted  by $\gr(k\,\_\,)$ and $\gr(k\,\_\,, \triv)$. Based on Section 3 of Webb's paper \cite{webb}, 
we suggest that the concept of primordial group can be defined for any globally defined Mackey functor. 
In terms of our two functors this concept is expressed in a familiar way: let $M$ be either $\gr(k\,\_\,)$ or $\gr(k\,\_\,, \triv)$. Then 
a group $G$ is called primordial for $M$ if $M(G)/T(G)\neq 0$ with
\begin{displaymath}
 T(G)=\sum_{\substack{H\hookrightarrow G\\ H\ncong G}}tr_H^GM(H).
\end{displaymath}
This definition also works  for the Mackey functor $G_0(k\,\_\,)$, which assigns to $G$ the Grothendieck group of finitely generated $kG$-modules. We write Prim$(M)$ for the class of primordial groups for $M$.

Primordial groups were first studied by Dress \cite{dress}  
in the context of Green functors for a finite group $G$. Th\'evenaz \cite{thev}  proved that for such a functor $N$, the closure under conjugation and subgroups of the primordials for $N$ is the minimal set $\mathscr D$ of subgroups of $G$ satisfying $N(G)=\sum_{H\in \mathscr D}tr_H^GN(H)$. 
Th\'evenaz also proved that if $k$ is a field of characteristic $p>0$, the primordial groups for $\mathbb Q\otimes \gr(kG)$ are the $p$-hypoelementary subgroups of $G$ (see Section 2 for definitions). Also, he conjectured that the primordials for $G_0(kG)$ were the $k$-elementary subgroups of $G$, which was proved in 1989 by Raggi \cite{raggi}, and that an analogue of the Brauer-Berman-Witt Theorem (5.6.7, in Benson \cite{benson}) should hold for $\gr(kG)$. The main result of this work, Theorem \ref{thethm}, states that this is true for the subring of trivial source modules.


Recent work in the subject of induction can be found in the papers of Boltje 
\cite{boltje} 
and  Co\c skun 
\cite{coskun}.
It is important to mention that  some results about primordial groups can be generalized to the context of globally defined Mackey functors, as for example Lemma \ref{prims} that shows a general behavior of primordial groups of subfunctors. 

\section{Preliminaries}

From now on, we assume that $k$ is a field of characteristic $p>0$,  all modules are finitely generated, and all groups are finite.

Recall that for a group $G$ and a prime $r$, $O^r(G)$ is defined as the smallest normal subgroup of $G$, such that $G/O^r(G)$ is a $r$-group, and $O_r(G)$ is the largest normal $r$-subgroup of $G$.

\begin{defi} If $q$ and $r$ are primes, a group $H$ is called $q$-\textit{hyperelementary} if $O^q(H)$ is cyclic,
 and $r$-\textit{hypoelementary} if $H/O_r(H)$ is cyclic.
\end{defi}
Observe that $H$ is $q$-hyperelementary if and only if $H=C\rtimes Q$  with $Q$ a $q$-group and $C$ a cyclic group of order prime to $q$, and it is $r$-hypoelementary if and only if $H=D\rtimes C$ where $D$ is a $r$-group and $C$ is cyclic of order prime to $r$. It is easy to prove that the classes of $q$-hyperelementary and $r$-hypoelementary groups are closed under subgroups and quotients.

\begin{nota}
 \label{res}
We write $\mathbb Z_m^*$ for the smallest non negative representatives of the multiplicative group of units modulo $m$
(which we denote by $(\zze/m\zze)^*$).
\end{nota}

\begin{defi}
 \label{defkele}
Suppose $H=C\rtimes Q$ is a $q$-hyperelementary group with \mbox{$C=<x>$} of order $m$ prime to $p$. The group $H$ is called \textit{$k$-elementary} if the action of every $y\in Q$ on $x$ is given by $yxy^{-1}=x^a$ with $a\in I_m(k)$, where $I_m(k)\subseteq \zze_m^*$ is the set of smallest non negative representatives of the image of $\gal(k(\omega)/k)$ under the injective morphism
\begin{eqnarray*}
 \gal(k(\omega)/k) & \longrightarrow & (\zze/m\zze)^*\\
         \sigma    & \longmapsto         &   \ov{a}
\end{eqnarray*}
if $\sigma (\omega)=\omega^a$ with $1\leq a\leq m-1$, $(a, m)=1$ and $\omega$ is a primitive $m$-th root of unity.

We write $\mathcal{E}_k$ for the class of $k$-elementary groups.
\end{defi}

Note that in the latter definition we can replace $I_m(k)$ by $I_n(k)$ where $n$ is any multiple of $m$.

\begin{defi}
\label{defdress}
For a prime $q$, a group $H$ is called $q$-\textit{Dress} if $O^q(H)$ is $p$-hypoelementary.

It is easy to see that $O^q(H)$ is $p$-hypoelementary if and only if $H/O_p(H)$ is $q$-hyperelementary. 
If in addition to this, $H/O_p(H)$ is $k$-elementary, then $H$ is called \textit{$k$-Dress} for $q$.
\end{defi}

The class of $q$-Dress groups is closed under subgroups and quotients, and it is denoted by $Dr_q$. The class of groups which are $k$-Dress for some prime will be denoted by $Dr_k$.
We write $Dr_p^*$ for the class of $k$-Dress groups such that $p$ divides the order of $H/O_p(H)$; that is, $H/O_p(H)=C\rtimes D$  is $k$-elementary with $D$ being a nontrivial $p$-group.

\begin{nota}
Let $G$ be a group and  $H$ a subgroup of $G$. If $M$ is a $kH$-module, we will write $M\!\ind{H}{G}$ for the induced module
 $kG\otimes_{kH}M$. If $N$ is a $kG$-module, the restriction of $N$ to $kH$ will be denoted by $N\!\res{H}{G}$.

\end{nota}

\subsection{Trivial source modules}

The following facts on trivial source modules are well known and can be found in Benson \cite{benson} and Curtis and Reiner \cite{curtis}.

Recall that for a $kG$-module it is equivalent  to be $(G,\, H)$-projective and to be a direct summand of $L\!\uparrow_H ^G$ for some $kH$-module $L$. If $M$ is an indecomposable $kG$-module, there exists a subgroup $D$ for which $M$ is $(G,\, D)$-projective and $D\leqslant_GH$ for any  subgroup $H$ for which $M$ is $(G,\, H)$-projective, such a group $D$ is called a \textit{vertex} of $M$. In this case, if $L$ is an indecomposable $kD$-module such that $M$ is a direct summand of $L\!\ind{D}{G}$, then  $L$ is called a \textit{source} of $M$. The module $M$ is said to have \textit{trivial source} if the field $k$ is a  source of $M$, which is equivalent to say that $M$ is a direct summand of a permutation module.

It can  be proved that any two vertices of $M$ are conjugate in $G$, and that any two sources of $M$ are conjugate by an element in $N_G(D)$. Since we are assuming that $k$ is a field of characteristic $p$, a vertex of $M$ is a $p$-subgroup of $G$.

We prove the following property of trivial source modules, which will be used later.

\begin{lemma}
 \label{tensorind}
 The tensor induction of a trivial source module is  a trivial source module.
\end{lemma}
\begin{proof}
We denote the tensor induction from $H$ to $G$ by $\,{\ind{H}{G}}^{\otimes}$. Let $B$ be a $kH$-module of trivial source, then  it is a direct summand of a permutation module, say $\oplus_{a\in [H/K]}k=B\oplus A$. From the proof of Proposition 3.15.2 $iii)$ in  Benson \cite{benson}, it is easy to see  that the tensor induction of a permutation module is  a permutation module. By the same proposition, on the right hand side we obtain \mbox{$B\,{\ind{H}{G}}^{\otimes} \oplus A\,{\ind{H}{G}}^{\otimes} \oplus X$}, where $X$ is a sum of modules induced from proper subgroups of $G$. 
\end{proof}

The following corollary to the Green Indecomposability Theorem, 19.22 in Curtis and Reiner \cite{curtis}, will be used in the following sections, as well as the next lemma. Recall that $k$ is a field of characteristic $p$.

\begin{coro}[19.23 in Curtis and Reiner \cite{curtis}]
 \label{green}
Suppose that $G$ is a $p$-group and $H$ is an arbitrary subgroup. If $L$ is an absolutely indecomposable $kH$ module (i.e $k'\otimes L$ is indecomposable for all $k'$ field extension of $k$), then $L\!\uparrow_H^G$ is an absolutely indecomposable $kG$-module.
\end{coro}

\begin{lemma}
\label{del2611}
Let $U$ be an indecomposable $kH$-module of trivial source with vertex containing $O_p(H)$, and let $M$ be an indecomposable $kG$-module for $G\leqslant H$ such that $U$ is a summand of $M\ind{G}{H}$. Suppose also that
if $|H/O_p(H)|$ is divisible by $p$, we have $M$  of trivial source. Then 
\begin{itemize}
\item [i)] $O_p(H)$ is contained in a vertex of $M$ and it acts trivially  on $M$.
\item [ii)] Any indecomposable summand $V$ of $M\ind{G}{H}$ has trivial source, $O_p(H)$ is contained in a vertex of $V$ and it acts trivially  on $V$.
\end{itemize} 
\end{lemma}
\begin{proof}
Let $D$ be a vertex of $U$ that contains $O_p(H)$. If $D_1$ and $S$ are a vertex and a source of $M$, respectively, then $O_p(H)\subseteq D_1\subseteq G$ (because $U$ is a direct summand of $S\ind{D_1}{H}$).
If $|H/O_p(H)|$ is a $p'$-number then $O_p(H)$ is a vertex of $U$ and $D_1\subseteq O_p(H)$, so we have $O_p(H)=D_1$. This implies that $S$ is a source of $U$ and that $M$ has trivial source. With this we obtain that $V$ is of trivial source, which is the first part of $ii)$.
  
Since $M$ is a summand of $k\!\uparrow_{D_1}^G$, then  $M\!\!\downarrow_{O_p(H)}^G$ is a direct summand of $k\!\uparrow_{D_1}^G\downarrow_{O_p(H)}^G$. By the Mackey formula, the latter is isomorphic to $\bigoplus_a k$, where $a$ runs over $[G/D_1]$ so $M\!\!\downarrow_{O_p(H)}^G$ is isomorphic to a sum of $k$.
With this we prove $i)$, and by the same argument, we prove that $O_p(H)$ acts trivially on $V$.

Finally, if $A$ is a vertex of $V$ then $V\!\!\!\res{O_p(H)}{H}\cong \bigoplus k$ is a summand of $k\ind{A}{H}\res{O_p(H)}{H}$, which is isomorphic to $\bigoplus_b k\!\ind{O_p(H)\cap A}{O_p(H)}$. By Corollary \ref{green}, we have that  $k\!\ind{O_p(H)\cap A}{O_p(H)}$ is indecomposable, then for some $b$ we should have $k\cong k\ind{O_p(H)\cap A}{O_p(H)}$, so $O_p(H)$ is contained in $A$.
\end{proof}

\section{Primordial groups}

\begin{lemma}
\label{prims} 
$ $ 
\begin{itemize}
\item[i)] Prim$(\gr(k\,\_\,))$ and Prim$(\gr(k\,\_\,, \triv))$ are closed under subgroups and quotients.
\item[ii)]$\mathcal{E}_k\subseteq\textrm{Prim}(\gr(k\,\_\,))\subseteq \textrm{Prim}(\gr(k\,\_\,, \triv))\subseteq \bigcup_qDr_{q}$.
\end{itemize}
\end{lemma}
\begin{proof}
 $i)$ The proof is the same for both functors, so $M$ represents any of them. 
First, let $G$ be a primordial group for $M$, and $H$ a subgroup of $G$. By \ref{tensorind} tensor induction provides a map from $M(H)$ to $M(G)$, and clearly it sends the class of the field $k$ to itself. Suppose that $k$ can be written as a linear combination of modules induced from proper subgroups of $H$, then, by $iii)$ and $iv)$ of Proposition 3.15.2  in Benson \cite{benson}, its image is a linear combination of modules induced from proper subgroups of $G$. This contradicts  that $G$ is primordial for $M$.

Now we take $G/K$, a quotient of $G$. Consider the inflation from $M(G/K)$ to $M(G)$. Again, the class of  $k$ is invariant under inflation, so if it could be written as a linear combination of modules induced from proper subgroups of $G/K$, then, these could be seen as modules induced form proper subgroups of $G$, which is a contradiction.

$ii)$ In order to prove the inclusions
\begin{displaymath}
 \mathcal{E}_k\subseteq\textrm{Prim}(\gr(k\,\_\,))\subseteq \textrm{Prim}(\gr(k\,\_\,, \triv))
\end{displaymath}
 we recall that for every group $H$, we have the following morphisms
\begin{displaymath}
 \gr(kH)\rightarrow G_0(kH)\quad \textrm{and}\quad \gr(kH,\,\triv)\hookrightarrow \gr(kH),
\end{displaymath}
the first one sends the class $[T]$ in $\gr(kH)$ to the class of $T$ in $G_0(kH)$. These are morphisms of unitary algebras and commute with induction. To represent any of them we write $f_H:M(H)\rightarrow N(H)$.
Now suppose that $H$ is primordial for $N$. Given the properties of $f_H$, if $k$ can be written as a linear combination of modules induced from proper subgroups of $H$ in $M(H)$, then that can also be made in $N(H)$, which is a contradiction. So $H$ is primordial for $M$. Recall that Prim$(G_0(k\,\_\,))=\mathcal{E}_k$, as mentioned in the Introduction.

To prove the inclusion $\textrm{Prim}(\gr(k\,\_\,, \triv))\subseteq \bigcup_qDr_{q}$, we will write $M$ for $\gr(k\,\_\,,\, \triv)$ and $\mathcal{D}$  for Prim$(\gr(k\,\_\,, \triv))$. Let $G$ be any group. We will use the following facts: 
\begin{itemize}
\item[a)] Using the Dress induction theorem for the Burnside ring, as stated in Yoshida's paper \cite{yoshida}, we obtain a generalization of this theorem for the Mackey functor $M$. Let $p$-Hypo be the class of $p$-hypoelementary groups, we have:
\begin{displaymath}
 M(G)=\sum_{\substack{K\leq G\\ K\in \mathcal{H}(p\textrm{-Hypo})}}\!\! tr_K^GM(K) + \bigcap_{\substack{L\leq G\\ L \in p\textrm{-Hypo}}}\textrm{ker}(res_L^G)
\end{displaymath}
where  $res_L^G$ is the restriction map from $M(G)$ to $M(L)$ and
\begin{displaymath}
\mathcal{H}(p\textrm{-Hypo})=\{H\ \textrm{group} \mid \exists q\ \textrm{prime with}\ O^q(H)\in p\textrm{-Hypo} \}. 
\end{displaymath}
It is not hard to prove that $\mathcal{H}(p\textrm{-Hypo})=\bigcup_q Dr_q$.
\item[b)] Observe that the Mackey functor $\gr(k\,\_\,)$ satisfies the Frobenius reciprocity formulas
\begin{displaymath}
tr_K^G(m\cdot res_K^Gn)=(tr_K^Gm)\cdot n\ \ \textrm{and}\ \ tr_K^G((res_K^Gn)\cdot m)=n\cdot (tr_K^Gm)
\end{displaymath}
for all $m$ in $\gr(kK)$ and $n$ in $\gr(kG)$, with $K\leq G$. 
\item[c)] Note also that $\mathcal{D}$  is the smallest class of groups closed under subgroups and quotients such that for every group $G$,
\begin{displaymath}
M(G)=\sum_{\substack{K\leq G\\ K\in \mathcal{D}}} tr_K^GM(K). 
\end{displaymath}
A proof of this is a slight modification of Th\'evenaz's in  \cite{thev}. 
\end{itemize}

We consider the inclusion $\gr(kG,\,\triv)\subseteq \mathbb Q\otimes \gr(kG)$ 
and we will write $N$ for $\mathbb Q\otimes \gr(k\,\_\,)$. From the article of Th\'evenaz \cite{thev}, we have that \mbox{Prim$(N)$} is the class $p$-Hypo and that
\begin{displaymath}
N(G)=\sum_{\substack{K\leq G\\ K\in p\textrm{-Hypo}}}tr_K^GN(K). 
\end{displaymath}
So we have 
\begin{displaymath}
1_{M(G)}=1_{N(G)}=\sum_{\substack{K\leq G\\ K\in p\textrm{-Hypo}}}tr_K^Gn_K 
\end{displaymath}
where $1_{M(G)}$ represents the unity of $M(G)$ and $n_K$ is an element of $N(K)$.
From the formula in $a)$, if $m$ is any element $m$ of $\bigcap_{\substack{L\leq G\\ L \in p\textrm{-Hypo}}}\textrm{ker}(res_L^G)$, we have
\begin{eqnarray*}
 m=1_{M(G)}\cdot m & = & \sum_{\substack{K\leq G\\ K\in p\textrm{-Hypo}}}(tr_K^Gn_K)\cdot m\\
                   & = & \sum_{\substack{K\leq G\\ K\in p\textrm{-Hypo}}}tr_K^G(n_K\cdot res_K^Gm)\quad \textrm{by $b)$}\\
                   & = & \sum_{\substack{K\leq G\\ K\in p\textrm{-Hypo}}}tr_K^G(n_K\cdot 0)=0.
\end{eqnarray*}
So we have
\begin{displaymath}
 M(G)=\sum_{\substack{K\leq G\\ K\in \mathcal{H}(p\textrm{-Hypo})}}\!\! tr_K^GM(K). 
\end{displaymath}
Since $\mathcal{H}(p\textrm{-Hypo})$ is closed under subgroups and quotients, using $c)$ we obtain $\textrm{Prim}(\gr(k\,\_\,, \triv))\subseteq \bigcup_qDr_{q}$.
\end{proof}


\begin{thm}
 \label{thethm}
$\pp(\gr(k\,\_\,,\, \triv))=Dr_k$.
\end{thm}

The proof is given by Propositions \ref{dosa} and \ref{uno}. 

With Proposition \ref{uno} we will conclude that every primordial group for $\gr(k\,\_\,)$ has to be $k$-Dress for some prime. On the other hand, the following proposition shows that every $k$-Dress group for a prime different from $p$ is primordial for $\gr(k\,\_\,)$.
As for the general case, the main difficulty arises from the fact that the techniques we use (namely Lemma \ref{del2611}) are ineffective for non trivial source modules.

\begin{pro}
 \label{dosa}
$ $
\begin{itemize}
\item [i)] $Dr_k\smallsetminus Dr_p^*\subseteq \pp(\gr(k\,\_\,))$.
\item [ii)] $Dr_p^*\subseteq \pp(\gr(k\,\_\,,\, \triv))$.
\end{itemize}
\end{pro}
\begin{proof}
We prove $i)$ and $ii)$ simultaneously by contradiction. Let $H$ be a $k$-Dress group. 
We have two cases:
\begin{itemize}
 \item $H/O_p(H)$ is not divisible by $p$; in this case we suppose $H$ is not primordial for $\gr(k\,\_\,)$ to prove $i)$.
\item $H/O_p(H)$ is divisible by $p$, so we assume $H$ is not primordial for $\gr(k\,\_\,,\, \triv)$ to prove $ii)$.
\end{itemize}
In both cases, $k$ can be written as a linear combination of modules induced from proper subgroups of $H$ 
 \begin{displaymath}
 k\oplus \Big(\bigoplus_i M_i\ind{L_i}{H}\Big )\cong \bigoplus_j N_j\ind{T_j}{H}.
\end{displaymath}
We show that in the first case we can assume that $M_i$ and $N_j$ are trivial source modules. Notice that if an indecomposable module of trivial source is a direct summand of one $M_i\ind{L_i}{H}$ (or $N_j\!\ind{T_j}{H}$), then by Lemma \ref{del2611}, $M_i$ (or $N_j$) and 
all the indecomposable summands are trivial source modules. Therefore, by the Krull-Schmidt Theorem, we can assume all of the $M_i$ and $N_j$ are of trivial source. Since in the second case
 we already assume that they are trivial source modules, the following arguments are valid for both cases. From \ref{del2611}, we have that $O_p(H)$ acts trivially on $N_j$ and $M_i$ and
that they have a vertex containing $O_p(H)$. We take the quotients
\begin{displaymath}
 k\oplus \Big( \bigoplus_i M_i\ind{L_i/O_p(H)}{H/O_p(H)}\Big )\cong \bigoplus_j N_j\ind{T_j/O_p(H)}{H/O_p(H)}.
\end{displaymath}
This isomorphism is an equality in $\gr(k(H/O_p(H)),\, \triv)$,  which is contained in $\gr(k(H/O_p(H)))$. 
Since $H/O_p(H)$ is $k$-elementary, Lemma \ref{prims} yields a contradiction.
\end{proof}

\begin{lemma}
 \label{notk}
If $H$ is of smallest order  
being $q$-Dress and not $k$-Dress, then $H$ is of the form
\begin{displaymath}
 H=<x>\rtimes <y>\ \textrm{where}\ \ |<x>|=r,\quad |<y>|=q^n
\end{displaymath}
 with $r$ and $q$ different primes and $yxy^{-1}=x^a$ with $a\in \zze_r^*\smallsetminus I_r(k)$.
\end{lemma}
\begin{proof}
Being a quotient of $H$, $H/O_p(H)$ is $q$-Dress 
but, since it is not $k$-elementary and $O_p(H/O_p(H))=1$, then $H/O_p(H)$ is not $k$-Dress. Therefore the minimality of $H$ implies $O_p(H)=1$. Hence,
 we have $H=C\rtimes Q$ with $C=<s>$ cyclic of order $m$ and $Q$ a $q$-group such that $m$ not divisible by $p$ and $q$. Now, as $H$ is not $k$-elementary, there exists $y\in Q$ such that $ysy^{-1}=s^a$ with $a\in \zze_m^*\smallsetminus I_m(k)$, so $H=C\rtimes C_{q^n}$, where $C_{q^n}$ is the cyclic of order $q^n$ generated by $y$.

Now we write $C=\prod_i C_{r_i}$, where $C_{r_i}$ is the $r_i$-Sylow subgroup of $C$ and the $r_i$ are all the primes that divide $m$, so
\begin{displaymath}
C\rtimes C_{q^n}=\prod_i (C_{r_i}\rtimes C_{q^n}).
\end{displaymath}
In addition, if $\omega$ is a primitive $m$-th root of unity, we have the commutative diagram
\[
  \xymatrix{
  \mathcal Gal(k(\omega)/k) \ar[d]_{\cong} \ar@{^{(}->}[r] & (\zze/m\zze)^*\ar[d]^{\cong} \\
  \prod_i \mathcal Gal(k(\omega^{m/r_i^{\alpha_i}})/k) \ar@{^{(}->}[r]^-{}  & \prod_i(\zze/r_i^{\alpha_i}\zze)^* }
\]
where $\alpha_i$ is the largest positive integer such that $r_i^{\alpha_i}$ divides $m$. Thus, there exists $i$ such that $C_{r_i}\rtimes C_{q^n}$ is not $k$-elementary. Rewriting $C_{r_i}=C_{r^{\alpha}}$ we have $H=C_{r^{\alpha}}\rtimes C_{q^n}$, where $C_{r^{\alpha}}=<x_o>$, $C_{q^n}=<y>$ and $yx_oy^{-1}=x_o^a$ with $a\in \zze_{r^{\alpha}}^*\smallsetminus I_{r^{\alpha}}(k)$.

Finally, we take $\zeta$ a primitive $r^{\alpha}$-th root of unity. We have 
\begin{displaymath}
 \gal(k(\zeta)/k)\cong \gal(k(\zeta)/k(\zeta^{r^{\alpha -1}}))\times \gal(k(\zeta^{r^{\alpha-1}})/k)
\end{displaymath}
and $(\zze/r^{\alpha}\zze)^*\cong A_{r^{\alpha -1}}\times A_{r-1}$ where these groups have order $r^{\alpha -1}$ and $r-1$ respectively. The  morphism 
\[
\xymatrix{
 \gal(k(\zeta)/k)\ar@{^{(}->}[r] & (\zze/r^{\alpha}\zze)^*}
\]
takes $\gal(k(\zeta)/k(\zeta^{r^{\alpha -1}}))$ into $A_{r^{\alpha -1}}$ and $\gal(k(\zeta^{r^{\alpha-1}})/k)$ into $A_{r-1}$, which is isomorphic to $(\zze/r\zze)^*$, so we have the commutative diagram
\[
\xymatrix{
  \mathcal Gal(k(\zeta)/k) \ar@{->>}[d] \ar@{^{(}->}[r] & (\zze/r^{\alpha}\zze)^*\ar@{->>}[d] \\
  \mathcal Gal(k(\zeta^{r^{\alpha-1}})/k) \ar@{^{(}->}[r]  & (\zze/r\zze)^*. }
\]
Now, we have $a^{q^n}\equiv 1$ mod $r^{\alpha}$, thus $r$ does not divide the order of $a$ modulo $r^{\alpha}$, and we have $a\in \zze_r^*$. Since a is not in $I_{r^{\alpha}}(k)$, we have $a\in\zze_r^*\smallsetminus I_r(k)$. Taking $x=x_o^{r^{\alpha -1}}$  
we have the result.
\end{proof}

\begin{pro}
 \label{uno}
$\pp(\gr(k\,\_\, ,\, \textnormal{triv}))\subseteq Dr_k$.
\end{pro}
\begin{proof}
 The proof is by contradiction. We suppose there is a group $H$ of smallest order in $\pp(\gr(k\,\_\, ,\, \triv))$ that is not $k$-Dress.
Observe that the previous lemma is also valid if $H$ satisfies a property 
that is preserved under subgroups and quotients and implies being $q$-Dress.
Since this is the case for the property of being primordial for $\gr(k\,\_\, ,\, \triv)$, the lemma gives us 
$H=C\rtimes Q$ with $C=<x>$ of order $r$ and  $Q=<y>$ of order $q^n$ with 
$r$ and $q$ different primes and $yxy^{-1}=x^a$ with $a\in \zze_r^*\smallsetminus I_r(k)$.

We write $1_H$ to identify the field $k$ as a $kH$-module. We shall prove that $1_H$ is a sum of modules induced from proper subgroups of $H$, contradicting the assumption of primordiality.

The image of $1_Q$ under the induction morphism $1_Q\!\ind{Q}{H}$ is isomorphic, as vector spaces, to $\bigoplus_{i=0}^{r-1}kx^i\otimes_Q1_Q$. 
If $\omega$ is a primitive $r$-th root of unity, 
then the $k(\omega)H$-module $k(w)\otimes 1_Q\ind{Q}{H}$ as $k(w)$-vector space has the basis 
\begin{displaymath}
\{x^i\otimes_Q 1_Q 
\mid i=0,\ldots, r-1 \}. 
\end{displaymath}
 We can define another basis
 $y_t:=\sum_{i=0}^{r-1}\omega^{-ti}(x^i\otimes_Q 1_Q)$
for $t=0,\ldots, r-1$. To prove it is a basis, observe that the matrix 
\begin{displaymath}
 \left(\begin{array}{cccccc}
          1 & 1            & \ldots & 1                & \ldots & 1\\
          1 & \omega^{-1}  & \ldots & \omega^{-j}      & \ldots & \omega^{-(r-1)}\\
     \vdots & \vdots       & \ddots & \vdots           & \ddots & \vdots\\
       1 & \omega^{-(r-1)} & \ldots & \omega^{-j(r-1)} & \ldots & \omega^{-(r-1)^2}
\end{array}\right)
\end{displaymath}
has determinant equal to $\prod_{i\neq j}(\omega^{-i}-\omega^{-j})$, which is different from $0$ since $\omega$ is a $r$-th primitive root of unity. 

$H$ acts on this basis in the following way
\[
\begin{array}{rl}
 xy_t & = \sum\limits_{i=0}^{r-1}\omega^{-ti}(x^{i+1}\otimes_Q 1_Q)\quad \textrm{if j=i+1,}\\
      & = \sum\limits_{j=0}^{r-1}\omega^{-t(j-1)}(x^j\otimes_Q 1_Q) = \omega^ty_t\\
 yy_t & = \sum\limits_{i=0}^{r-1}\omega^{-ti}(yx^i\otimes_Q 1_Q)=\sum\limits_{i=0}^{r-1}\omega^{-ti}(x^{ai}\otimes_Q 1_Q)\\
      & = \sum\limits_{i=0}^{r-1}\omega^{-tbi}(x^i\otimes_Q 1_Q)= y_{t'}
\end{array}
\]
where $b\in \zze_r^*$ is such that $\ov{b}\ov{a}=1$ in $(\zze/r\zze)^*$, and $0\leq t'\leq r-1$ satisfies $t'\equiv tb$ mod $r$.
From these relations we see that $y_0$ is fixed under the action of $H$, so $k(\omega)y_0$ is $k(\omega)H$-isomorphic to $k(\omega)$ and we have
\begin{displaymath}
 k(\omega)\otimes 1_Q\ind{Q}{H}\cong k(\omega)\oplus \Bigg( \sum_{t=1}^{r-1}k(\omega)y_t \Bigg).
\end{displaymath}
It is clear that $\mathcal G=\gal (k(\omega)/k)$ acts on $k(\omega)\otimes 1_Q\ind{Q}{H}$. Taking the fixed points of this action in the isomorphism above,  gives us
\begin{displaymath}
 1_Q\ind{Q}{H}\cong 1_H\oplus \Bigg( \sum_{t=1}^{r-1}k(\omega)y_t \Bigg)^{\mathcal G}.
\end{displaymath}

Let $\sigma$ be in $\mathcal G$, we write $b_{\sigma}\in I_{r}(k)$ for the integer through which $\sigma$ is defined. We have
\begin{displaymath}
 \sigma y_t= \sum_{i=0}^{r-1}\sigma (\omega^{-ti})(x^i\otimes 1)=\sum_{i=0}^{r-1}\omega^{-tb_{\sigma}i}(x^i\otimes 1)=y_{s}
\end{displaymath}
where $0\leq s\leq r-1$ and $s\equiv tb_{\sigma}$ mod $r$. 
If $u=\sum_{t=1}^{r-1}\lambda_ty_t$ is in $( \sum_{t=1}^{r-1}k(\omega)y_t )^{\mathcal G}$, then for each $t\in \zze_r^*$ we must have $\sigma(\lambda_t)=\lambda_{s}$, where $s\in \zze_r^*$ and $s\equiv tb_{\sigma}$ mod $r$.
From this we define the vector spaces
\begin{displaymath}
M_l:=\Bigg\{ \sum_{\sigma \in \mathcal G} \sigma(\lambda_l)y_s\,\Big|\, s\equiv lb_{\sigma}\ \textrm{mod}\ r,\, \lambda_{l}\in k(\omega)\Bigg \}
\end{displaymath}
for each $l\in \zze_r^*$. Observe that $M_{l_1}=M_{l_2}$ if and only if $l_1\equiv l_2b_{\sigma}$ mod $r$ for some $\sigma \in \mathcal G$, this implies that 
\begin{displaymath}
 \Bigg( \sum_{t=1}^{r-1}k(\omega)y_t \Bigg)^{\mathcal G}= \bigoplus_{\substack{l\in \zze_r^*/I_r(k)}}M_l.
\end{displaymath}
We shall prove that the right hand side of this equality is a sum of modules induced from proper subgroups of $H$. We have $xM_l=M_l$ and $yM_l=M_{l'}$ with $l'\in \zze_r^*$ and $l'\equiv lb$ mod $r$. Since $a$ does not belong to $I_r(k)$, neither does $b$, so $M_l$ is never fixed under the action of $y$. Then $M_l$ is not a $kH$-module. Taking the orbits of the action of $y$ we have
\begin{eqnarray*}
 \Bigg( \sum_{t=1}^{r-1}k(\omega)y_t \Bigg)^{\mathcal G} & = & \bigoplus_{\substack{l\in \frac{\zze_r^*/I_r(k)}{\sim}}}\Big(                                                                 \bigoplus _{z\in[H/A]}zM_l\Big)\\
                                                         & = & \bigoplus_{\substack{l\in \frac{\zze_r^*/I_r(k)}{\sim}}} \big(M_l\ind{A}{H}\big)
\end{eqnarray*}
where $A=Stab_H(M_l)$, and $\sim$  represents the action of $y$. 

It is clear that $A$ is a proper subgroup of $H$. Finally, observe that every $M_l$ is a trivial source $kA$-module. If $A=C_r\rtimes<y^d>$, then $M_l$ is a direct summand of the induced module $\big (\sum_{\sigma \in \mathcal G}ky_s\big )\ind{<y^d>}{A}$ , where $s\in \zze_r^*$ and $s\equiv lb_{\sigma}$ mod $r$. Since $ky_s\cong k$, then $M_l$ has trivial source. 
\end{proof}

\section*{Acknowledgements}

We wish to thank  Julio Aguilar Cabrera, whose Master thesis was of great help in writing this article, and Luis Valero Elizondo for helpful conversations.

\bibliographystyle{plain}
\bibliography{prim6}

\end{document}